\documentclass[11pt,letterpaper,reqno]{amsart}

\usepackage{amsmath,amssymb,amsthm,amsfonts,mathtools}
\usepackage{mathrsfs}
\usepackage{xcolor}
\usepackage[T1]{fontenc}
\usepackage{microtype}
\usepackage{aliascnt}

\usepackage[colorlinks=true,
  linkcolor=blue,
  citecolor=blue,
  urlcolor=blue]{hyperref}
\usepackage[nameinlink,capitalize,noabbrev]{cleveref}

\allowdisplaybreaks
\numberwithin{equation}{section}

\newtheorem{theorem}{Theorem}[section]

\newaliascnt{problem}{theorem}
\newtheorem{problem}[problem]{Problem}
\aliascntresetthe{problem}

\newaliascnt{proposition}{theorem}
\newtheorem{proposition}[proposition]{Proposition}
\aliascntresetthe{proposition}

\newaliascnt{lemma}{theorem}
\newtheorem{lemma}[lemma]{Lemma}
\aliascntresetthe{lemma}

\newaliascnt{corollary}{theorem}

\aliascntresetthe{corollary}

\newaliascnt{remark}{theorem}
\newtheorem{remark}[remark]{Remark}
\aliascntresetthe{remark}

\crefname{theorem}{Theorem}{Theorems}
\Crefname{theorem}{Theorem}{Theorems}
\crefname{problem}{Problem}{Problems}
\Crefname{problem}{Problem}{Problems}
\crefname{proposition}{Proposition}{Propositions}
\Crefname{proposition}{Proposition}{Propositions}
\crefname{lemma}{Lemma}{Lemmas}
\Crefname{lemma}{Lemma}{Lemmas}
\crefname{corollary}{Corollary}{Corollaries}
\Crefname{corollary}{Corollary}{Corollaries}
\crefname{remark}{Remark}{Remarks}
\Crefname{remark}{Remark}{Remarks}

\newaliascnt{claim}{theorem}

\aliascntresetthe{claim}

\newaliascnt{question}{theorem}

\aliascntresetthe{question}

\newaliascnt{conjecture}{theorem}

\aliascntresetthe{conjecture}

\theoremstyle{definition}

\newaliascnt{definition}{theorem}

\aliascntresetthe{definition}

\newaliascnt{example}{theorem}

\aliascntresetthe{example}

\crefname{claim}{Claim}{Claims}
\Crefname{claim}{Claim}{Claims}
\crefname{question}{Question}{Questions}
\Crefname{question}{Question}{Questions}
\crefname{conjecture}{Conjecture}{Conjectures}
\Crefname{conjecture}{Conjecture}{Conjectures}
\crefname{definition}{Definition}{Definitions}
\Crefname{definition}{Definition}{Definitions}
\crefname{example}{Example}{Examples}
\Crefname{example}{Example}{Examples}

\DeclareMathOperator{\Span}{span}
\DeclareMathOperator{\diam}{diam}
\newcommand{\R}{\mathbb R}

\newcommand{\II}{\mathrm{II}}

\begin{document}

\title[Sharp Neumann eigenvalue ratios]{Sharp ratios for low-index Neumann eigenvalues on convex domains}

\author[Q.~Tang]{Quanyu Tang}
\address{School of Mathematics and Statistics, Xi'an Jiaotong University, Xi'an 710049, P. R. China}
\email{tangquanyu827@gmail.com}

\author[H.~Zhang]{Haiqi Zhang}
\address{School of Mathematics, Shandong University, Jinan 250100, P. R. China}
\email{ZHQAQ2024@outlook.com}



\begin{abstract}
Let $\Omega\subset\R^N$ be a bounded open convex set, and let $0=\mu_0(\Omega)<\mu_1(\Omega)\le \mu_2(\Omega)\le\cdots$ be the Neumann eigenvalues of the Laplacian, repeated according to multiplicity. We prove the sharp bounds
$$
\mu_2(\Omega)\le 4\mu_1(\Omega),\qquad
\mu_3(\Omega)\le 9\mu_1(\Omega).
$$
The first estimate resolves a problem attributed to Henrot, while the second gives the next sharp case predicted by the one-dimensional model. The constants are optimal in every dimension.
\end{abstract}

\maketitle

\section{Introduction}

Throughout the paper, \(\Omega\subset\R^N\) is a nonempty bounded open
convex set. Such a set is connected and has Lipschitz boundary. The Neumann
Laplacian on \(\Omega\) is understood through the closed quadratic form
\[
  w\mapsto \int_\Omega |\nabla w|^2\,dx, \qquad w\in H^1(\Omega).
\]
Since the embedding \(H^1(\Omega)\hookrightarrow L^2(\Omega)\) is compact,
the associated self-adjoint operator has compact resolvent. Its eigenvalues,
repeated according to multiplicity, are denoted by
\begin{equation}\label{eq:spectrum}
0=\mu_0(\Omega)<\mu_1(\Omega)\le \mu_2(\Omega)\le\cdots.
\end{equation}
Thus \(\mu_1\) is the first positive Neumann eigenvalue. The quotients
\(\mu_k(\Omega)/\mu_1(\Omega)\) are scale invariant. A natural guiding
problem is whether, under convexity, these quotients are controlled by the
one-dimensional model.

The first case of this question has been explicitly formulated in several
closely related forms. In particular, Ashbaugh included the problem of
bounding the ratio between the first two positive Neumann eigenvalues on
convex domains as Problem~10 in his list of open problems on Laplacian
eigenvalues \cite[Problem~10]{Ashbaugh1999Open}. The American Institute of
Mathematics problem list from the workshop \emph{Shape optimization with
surface interactions} records the estimate
\(\mu_2(\Omega)/\mu_1(\Omega)\le 4\) as Conjecture~3.2, attributed to
Henrot, together with the broader expected \(k^2\)-pattern for higher
Neumann eigenvalue ratios~\cite[Conjecture~3.2]{AIMProblem}. The same AIM entry
points to Ashbaugh--Benguria \cite{AshbaughBenguria1993} for the
two-dimensional formulation and to Antunes--Henrot
\cite{AntunesHenrot2011} for partial analytic progress in dimension two. In the notation of
\eqref{eq:spectrum}, the first case asks for the sharp estimate $\mu_2(\Omega)\le 4\mu_1(\Omega)$ for every bounded open convex set \(\Omega\subset\R^N\). 
We record this problem in the following form.

\begin{problem}\label{prob:ratio}
Prove that every nonempty bounded open convex set \(\Omega\subset\R^N\)
satisfies
\[
  \frac{\mu_2(\Omega)}{\mu_1(\Omega)}\le 4.
\]
\end{problem}

The broader conjectural pattern recorded in the same AIM entry \cite[Conjecture~3.2]{AIMProblem} is
\begin{equation}\label{eq:general-conjecture}
\frac{\mu_k(\Omega)}{\mu_1(\Omega)}\le k^2,\qquad k\ge2.
\end{equation}
This pattern is forced by the one-dimensional model: if \(I=(0,L)\), then the Neumann eigenvalues are
\[
\mu_k(I)=\frac{k^2\pi^2}{L^2},\qquad k=0,1,2,\ldots,
\]
and therefore equality holds in \eqref{eq:general-conjecture} on intervals.

This ratio problem belongs to the broader study of Neumann eigenvalue
bounds on convex domains. The first positive eigenvalue satisfies the
classical Payne--Weinberger inequality
\[
\mu_1(\Omega)\ge \frac{\pi^2}{\diam(\Omega)^2}
\]
\cite{PayneWeinberger1960}. Kr\"oger proved volume-dependent upper bounds for Neumann eigenvalues on
general bounded Euclidean domains \cite{Kroger1992}, and later obtained
sharp diameter-dependent upper bounds in the convex class
\cite{Kroger1999}. Henrot and Michetti subsequently extended this
diameter-constrained picture, via one-dimensional Sturm--Liouville
reductions, to a broader class of domains including convex domains
\cite{HenrotMichetti2024}.
In the broader direction of eigenvalue-ratio estimates, Liu proved a dimension-free \(O(k^2)\) upper bound for eigenvalue ratios
in nonnegative-curvature settings, with the order in \(k\) being optimal,
including closed weighted Riemannian manifolds with nonnegative
Bakry--\'Emery Ricci curvature and finite-dimensional Alexandrov spaces
of nonnegative curvature \cite{Liu2014}. Recent universal inequalities for Neumann eigenvalues on convex domains
were obtained by Funano
\cite{Funano2024,Funano2026Distribution,Funano2026Planar}.
Related shape-optimization problems for Neumann eigenvalues under
convexity, diameter, and perimeter constraints, including existence results
and numerical investigations, were studied by Bogosel--Henrot--Michetti
\cite{BogoselHenrotMichetti2024}.

We now state the two main results. The first gives the sharp estimate
for the second positive Neumann eigenvalue and resolves the convex Neumann
ratio problem in \Cref{prob:ratio}.

\begin{theorem}\label{thm:mu2}
Let \(N\ge1\), and let \(\Omega\subset\R^N\) be a nonempty bounded open
convex set. Then
\begin{equation}\label{eq:main-mu2}
  \mu_2(\Omega)\le 4\mu_1(\Omega).
\end{equation}
The constant \(4\) is optimal in every dimension.
\end{theorem}

The next theorem proves the corresponding sharp estimate for the third
positive Neumann eigenvalue.

\begin{theorem}\label{thm:mu3}
Let \(N\ge1\), and let \(\Omega\subset\R^N\) be a nonempty bounded open
convex set. Then
\begin{equation}\label{eq:main-mu3}
  \mu_3(\Omega)\le 9\mu_1(\Omega).
\end{equation}
The constant \(9\) is optimal in every dimension.
\end{theorem}

\begin{remark}
The present paper proves only the cases \(k=2\) and \(k=3\) of \eqref{eq:general-conjecture}; no assertion is made here for \(k\ge4\).
\end{remark}

\subsection{Proof strategy}

The proof is based on a simple idea whose implementation is rather rigid:
we use a first Neumann eigenfunction to generate low-dimensional test
spaces, and then prove weighted one-dimensional inequalities strong enough
to recover the sharp constants.

Let \(u\) be a first nonconstant Neumann eigenfunction,
\[
  -\Delta u=\lambda u,\qquad \partial_\nu u=0,\qquad
  \lambda=\mu_1(\Omega).
\]
For the estimate on \(\mu_2\), the multiplicity case is immediate. When
\(\mu_1\) is simple, we consider a quadratic transform $v=u^2-cu-d$ chosen so that \(v\) is orthogonal to constants and to the first eigenspace.
The spectral theorem then gives
\[
  \|-\Delta v\|_{L^2(\Omega)}^2
  \ge
  \mu_2(\Omega)\int_\Omega |\nabla v|^2\,dx.
\]
The required upper bound for the left-hand side is reduced to the fourth-order estimate
\[
  \int_\Omega |\nabla u|^4\,dx
  \le
  \lambda^2\left(
    \int_\Omega u^4\,dx
    -\frac34
     \frac{\left(\int_\Omega u^3\,dx\right)^2}{\int_\Omega u^2\,dx}
  \right).
\]
This estimate follows from the integrated Bochner identity. Convexity enters
precisely through the boundary sign
\[
  \partial_\nu |\nabla u|^2\le0
  \qquad\text{on }\partial\Omega,
\]
which is the Neumann form of the nonnegativity of the second fundamental
form.

The proof of the estimate on \(\mu_3\) uses a different mechanism. Instead
of a single quadratic test function, we use the four-dimensional space $\mathcal V=\Span\{1,u,u^2,u^3\}$.
For \(f=P(u)\), where \(\deg P\le3\), one has the exact identity
\begin{equation}\label{eq:intro-laplacian-identity}
  \|-\Delta P(u)\|_{L^2(\Omega)}^2
  =
  \lambda\int_\Omega P'(u)^2|\nabla u|^2\,dx
  +
  \int_\Omega P''(u)^2|\nabla u|^4\,dx.
\end{equation}
Thus the desired bound follows from the weighted polynomial inequality
\begin{equation}\label{eq:intro-weighted-quadratic}
  \int_\Omega q'(u)^2|\nabla u|^4\,dx
  \le
  8\lambda\int_\Omega q(u)^2|\nabla u|^2\,dx,
  \qquad \deg q\le2.
\end{equation}
The proof of \eqref{eq:intro-weighted-quadratic} is the main new point in
the cubic argument. It uses a weighted Bochner--Reilly estimate, a centered
moment bound for the measure \(|\nabla u|^4\,dx\), and a two-by-two matrix
positivity argument for an adjoint first-order operator. Once
\eqref{eq:intro-weighted-quadratic} is available, \eqref{eq:intro-laplacian-identity}
gives
\[
  \|-\Delta P(u)\|_{L^2(\Omega)}^2
  \le
  9\lambda\int_\Omega |\nabla P(u)|^2\,dx,
\]
and the min--max principle applied to \(\mathcal V\) yields
\(\mu_3(\Omega)\le9\mu_1(\Omega)\).

The estimates are first proved for smooth bounded convex domains. General
bounded open convex sets are obtained by approximating convex bodies by
smooth strictly convex bodies and using spectral continuity of Neumann
eigenvalues in the convex class.

\subsection{Paper organization}

\Cref{sec:mu2} proves \Cref{thm:mu2} for smooth bounded convex domains.
\Cref{sec:mu3} proves \Cref{thm:mu3} in the smooth case. Finally,
\Cref{sec:approx-sharp} passes to arbitrary bounded open convex sets and
proves the sharpness of both constants.

\section{The second positive eigenvalue: the smooth case}\label{sec:mu2}

Throughout this section \(N\ge2\), \(\Omega\subset\R^N\) is a smooth bounded
convex domain, and \(u\in C^\infty(\overline\Omega)\) is a nonzero solution
of
\begin{equation}\label{eq:neumann-eigenproblem}
  -\Delta u=\lambda u\quad\text{in }\Omega,
  \qquad
  \partial_\nu u=0\quad\text{on }\partial\Omega,
\end{equation}
with \(\lambda>0\). Here \(\nu\) denotes the outward unit normal and \(D\) denotes the
Euclidean derivative. We write \(D_X V\) for the
standard directional derivative of a vector field \(V\) in the direction
\(X\) in \(\R^N\); in coordinates, \(D_XV=(X\cdot\nabla)V\). Thus
\[
  D^2u(X,Y)=\langle D_X\nabla u,Y\rangle
\]
is the Hessian written as a bilinear form. We also identify \(D^2u\) with
the corresponding symmetric linear map when writing expressions such as
\(D^2u\,\nabla u\). On \(\partial\Omega\), \(\nabla_Tu\) denotes the
tangential gradient. For tangent vector fields \(X,Y\) on \(\partial\Omega\),
we use the second fundamental form convention
\[
  \II(X,Y)=\langle D_X\nu,Y\rangle.
\]
With this convention the unit sphere has \(\II(X,X)=|X|^2\), and convexity
means that \(\II\) is positive semidefinite, written \(\II\succeq0\). Set
\[
  g=|\nabla u|^2.
\]
The identities used below are standard consequences of the Bochner formula
and of the boundary terms appearing in Reilly-type integral formulas; see
\cite{Reilly1977}. Since sign conventions for the second fundamental form
vary in the literature, we keep the above convention throughout. We include
the short weighted integration because the sharp constant in the subsequent
argument depends on it.

\begin{lemma}\label{lem:bochner-mu2}
Let
\[
  G=\int_\Omega g^2\,dx,
  \qquad
  J=\int_\Omega |D^2u\,\nabla u|^2\,dx.
\]
Then
\begin{equation}\label{eq:J-bound}
  J\le \frac{\lambda}{3}G.
\end{equation}
\end{lemma}

\begin{proof}
If \(X\) is tangent to \(\partial\Omega\), differentiating the Neumann
condition tangentially gives
\[
  0=X(\partial_\nu u)=D^2u(X,\nu)+\langle \nabla u,D_X\nu\rangle.
\]
Since \(\nabla u\) is tangent to \(\partial\Omega\), taking
\(X=\nabla_Tu\) yields
\begin{equation}\label{eq:boundary-gradient-sign-mu2}
  \partial_\nu g
  =2D^2u(\nu,\nabla u)
  =-2\II(\nabla_Tu,\nabla_Tu)
  \le0
  \quad\text{on }\partial\Omega.
\end{equation}
The Bochner identity gives
\begin{equation}\label{eq:bochner-basic-mu2}
  \frac12\Delta g=|D^2u|^2-\lambda g.
\end{equation}
Since \(\nabla g=2D^2u\,\nabla u\), integration by parts, followed by
\eqref{eq:boundary-gradient-sign-mu2} and \eqref{eq:bochner-basic-mu2}, gives
\begin{align}
  4J
  &=\int_\Omega |\nabla g|^2\,dx \nonumber\\
  &=-\int_\Omega g\Delta g\,dx
    +\int_{\partial\Omega}g\partial_\nu g\,d\sigma \nonumber\\
  &\le 2\lambda G-2\int_\Omega g|D^2u|^2\,dx.
  \label{eq:integrated-bochner-mu2}
\end{align}
The pointwise inequality
\[
  |D^2u\,\nabla u|^2\le |D^2u|^2|\nabla u|^2=g|D^2u|^2
\]
implies
\[
  \int_\Omega g|D^2u|^2\,dx\ge J.
\]
Since this integral appears in \eqref{eq:integrated-bochner-mu2} with the
negative coefficient \(-2\), we obtain
\[
  4J\le 2\lambda G-2\int_\Omega g|D^2u|^2\,dx
  \le 2\lambda G-2J.
\]
Thus \eqref{eq:J-bound} follows.
\end{proof}
We next introduce moment identities and a fourth-order estimate.
\begin{lemma}\label{lem:fourth-order}
Set
\[
  S=\int_\Omega u^2\,dx,
  \qquad
  M_3=\int_\Omega u^3\,dx,
  \qquad
  M_4=\int_\Omega u^4\,dx.
\]
Then
\begin{equation}\label{eq:moment-identities-mu2}
  \int_\Omega g\,dx=\lambda S,
  \qquad
  \int_\Omega ug\,dx=\frac{\lambda}{2}M_3,
  \qquad
  \int_\Omega u^2g\,dx=\frac{\lambda}{3}M_4,
\end{equation}
and
\begin{equation}\label{eq:fourth-order-estimate}
  \int_\Omega |\nabla u|^4\,dx
  \le
  \lambda^2\left(
    M_4-\frac{3M_3^2}{4S}
  \right).
\end{equation}
\end{lemma}

\begin{proof}
Since \(u\not\equiv0\), we have \(S>0\). For \(m=1,2,3\), testing
\eqref{eq:neumann-eigenproblem} against \(u^m\) and using
\(\partial_\nu u=0\) gives
\[
  m\int_\Omega u^{m-1}|\nabla u|^2\,dx
  =\lambda\int_\Omega u^{m+1}\,dx,
\]
which proves \eqref{eq:moment-identities-mu2}.

Set
\[
  G=\int_\Omega g^2\,dx,
  \qquad
  \alpha=\frac{\int_\Omega ug\,dx}{\int_\Omega g\,dx}
  =\frac{M_3}{2S},
\]
and
\[
  Q=\int_\Omega (u-\alpha)^2g\,dx,
  \qquad
  B=\int_\Omega (u-\alpha)D^2u(\nabla u,\nabla u)\,dx,
  \qquad
  J=\int_\Omega |D^2u\,\nabla u|^2\,dx.
\]
By the definition of \(\alpha\),
\begin{equation}\label{eq:weighted-mean-zero-mu2}
  \int_\Omega (u-\alpha)g\,dx=0.
\end{equation}
The vector field \((u-\alpha)g\nabla u\) has zero normal component on
\(\partial\Omega\). Using \(\nabla g=2D^2u\,\nabla u\),
\(\Delta u=-\lambda u\), the divergence theorem, and
\eqref{eq:weighted-mean-zero-mu2}, we obtain
\begin{equation}\label{eq:G-B-Q-mu2}
  G=\lambda Q-2B.
\end{equation}
Moreover,
\begin{equation}\label{eq:B-cauchy-mu2}
  |B|\le \sqrt{QJ}
\end{equation}
by the Cauchy--Schwarz inequality. Combining \eqref{eq:G-B-Q-mu2},
\eqref{eq:B-cauchy-mu2}, and \cref{lem:bochner-mu2}, we find
\[
  G\le \lambda Q+2\sqrt{QJ}
  \le \lambda Q+2\sqrt{\frac{\lambda QG}{3}}.
\]
If \(Q=0\), then \eqref{eq:B-cauchy-mu2} gives \(B=0\), and
\eqref{eq:G-B-Q-mu2} gives \(G=0\), contradicting the nonconstancy of
\(u\). Thus \(Q>0\). With \(y=(G/(\lambda Q))^{1/2}\), the preceding
inequality becomes $y^2\le1+2y/\sqrt3$.
Since
\[
  y^2-\frac{2}{\sqrt3}y-1
  =\left(y-\sqrt3\right)\left(y+\frac1{\sqrt3}\right),
\]
we get \(y\le\sqrt3\), and therefore
\begin{equation}\label{eq:GQ-bound-mu2}
  G\le3\lambda Q.
\end{equation}
Finally, by \eqref{eq:moment-identities-mu2},
\[
  Q
  =\int_\Omega u^2g\,dx
  -\frac{\left(\int_\Omega ug\,dx\right)^2}{\int_\Omega g\,dx}
  =\lambda\left(\frac{M_4}{3}-\frac{M_3^2}{4S}\right).
\]
Substituting this into \eqref{eq:GQ-bound-mu2} proves
\eqref{eq:fourth-order-estimate}.
\end{proof}
The quadratic test function now gives the sharp estimate for \(\mu_2\) in
the smooth case.
\begin{proposition}\label{prop:mu2-smooth}
Let \(N\ge2\), and let \(\Omega\subset\R^N\) be a smooth bounded convex domain. Then $\mu_2(\Omega)\le4\mu_1(\Omega)$.
\end{proposition}

\begin{proof}
Set \(\lambda=\mu_1(\Omega)\). If \(\lambda\) has multiplicity at least
\(2\), then \(\mu_2(\Omega)=\mu_1(\Omega)\), and the assertion is immediate.
Assume that \(\lambda\) is simple. Let \(u\) be a real first Neumann
eigenfunction satisfying \eqref{eq:neumann-eigenproblem}. Integrating the
equation gives \(\int_\Omega u\,dx=0\). Keep the notation
\(S,M_3,M_4,g\) from \cref{lem:fourth-order}. Define
\[
  c=\frac{M_3}{S},
  \qquad
  d=\frac{S}{|\Omega|},
  \qquad
  v=u^2-cu-d.
\]
Then
\begin{equation}\label{eq:v-orthogonality}
  \int_\Omega v\,dx=0,
  \qquad
  \int_\Omega uv\,dx=0.
\end{equation}
Moreover, \(v\not\equiv0\). Indeed, if \(v\equiv0\), then the continuous
function \(u\) takes values only among the roots of \(t^2-ct-d\). Since
\(\Omega\) is connected, \(u\) would be constant, a contradiction. Thus
\begin{equation}\label{eq:v-positive-energy}
  \int_\Omega |\nabla v|^2\,dx>0.
\end{equation}

Let \(A=-\Delta_{\mathrm N}\) be the self-adjoint Neumann Laplacian on
\(L^2(\Omega)\), and denote its operator domain by \(\operatorname{Dom}(A)\). Since
\(v\in C^\infty(\overline\Omega)\) and
\[
  \partial_\nu v=(2u-c)\partial_\nu u=0,
\]
we have \(v\in\operatorname{Dom}(A)\). By \eqref{eq:v-orthogonality} and
simplicity of \(\lambda\), the function \(v\) is orthogonal to the
eigenspaces corresponding to \(\mu_0\) and \(\mu_1\). The spectral theorem
therefore gives
\begin{equation}\label{eq:second-order-spectral-quotient}
  \|Av\|_{L^2(\Omega)}^2
  \ge
  \mu_2(\Omega)\int_\Omega |\nabla v|^2\,dx.
\end{equation}
Set $R_0=4M_4/3-M_3^2/S$. Since \(\nabla v=(2u-c)\nabla u\), the moment identities
\eqref{eq:moment-identities-mu2} give
\begin{equation}\label{eq:v-energy}
  \int_\Omega |\nabla v|^2\,dx
  =\lambda\left(\frac43M_4-\frac{M_3^2}{S}\right)
  =\lambda R_0.
\end{equation}
In particular, \(R_0>0\). Using \eqref{eq:neumann-eigenproblem}, $Av=2\lambda u^2-2g-\lambda cu$.
Expanding the square and using \eqref{eq:moment-identities-mu2}, we obtain
\begin{equation}\label{eq:Av-norm}
  \|Av\|_{L^2(\Omega)}^2=4\int_\Omega g^2\,dx+\lambda^2R_0.
\end{equation}
The fourth-order estimate \eqref{eq:fourth-order-estimate} gives
\[
  4\int_\Omega g^2\,dx
  \le 4\lambda^2M_4-3\lambda^2\frac{M_3^2}{S}
  =3\lambda^2R_0.
\]
Together with \eqref{eq:Av-norm} and \eqref{eq:v-energy}, this gives
\[
  \|Av\|_{L^2(\Omega)}^2
  \le 4\lambda^2R_0
  =4\lambda\int_\Omega |\nabla v|^2\,dx.
\]
Combining this with \eqref{eq:second-order-spectral-quotient} and
\eqref{eq:v-positive-energy} yields $\mu_2(\Omega)\le4\lambda=4\mu_1(\Omega)$.
\end{proof}

\section{The third positive eigenvalue: the smooth case}\label{sec:mu3}

Throughout this section \(N\ge2\), \(\Omega\subset\R^N\) is a smooth bounded
convex domain, \(\lambda=\mu_1(\Omega)\), and \(u\) is a real first
nonconstant Neumann eigenfunction:
\[
  -\Delta u=\lambda u\quad\text{in }\Omega,
  \qquad
  \partial_\nu u=0\quad\text{on }\partial\Omega.
\]
Set
\[
  g=|\nabla u|^2,
  \qquad
  M=D^2u,
  \qquad
  W=M(\nabla u,\nabla u).
\]
On \(\{g>0\}\), put \(a=2W/g\), and put \(a=0\) on \(\{g=0\}\). Since
\(W=0\) whenever \(g=0\), we have everywhere in \(\Omega\)
\begin{equation}\label{eq:a-identities}
  ag=2W=\nabla g\cdot\nabla u.
\end{equation}
Moreover, on \(\{g>0\}\),
\[
  a^2g=\frac{4W^2}{g}\le4|M\nabla u|^2,
\]
and both sides vanish on \(\{g=0\}\). Also \(|a|\le2|M|\) on
\(\{g>0\}\), so this extension of \(a\) is bounded and measurable. No
derivative of \(a\) will be used.

For a measurable function \(F\), write
\[
  \|F\|_{L^2(g)}^2=\int_\Omega F^2g\,dx.
\]
The next identity is the integration-by-parts form used in the weighted
polynomial estimate. It may be viewed as an adjoint relation for the
first-order operation \(q\mapsto q'(u)\) with respect to the weights
\(g^2\,dx\) and \(g\,dx\).

\begin{lemma}\label{lem:adjoint}
Let \(q\) be a polynomial and let \(h\) be an affine function. Define
\[
  Rh:=\lambda u\,h(u)-g\,h'(u)-a\,h(u).
\]
Then
\begin{equation}\label{eq:adjoint}
  \int_\Omega q'(u)h(u)g^2\,dx
  =\int_\Omega q(u)Rh\,g\,dx.
\end{equation}
\end{lemma}

\begin{proof}
The vector field \(q(u)h(u)g\nabla u\) has zero normal component on
\(\partial\Omega\). Hence
\[
  0=\int_\Omega \operatorname{div}\bigl(q(u)h(u)g\nabla u\bigr)\,dx.
\]
Expanding the divergence gives
\begin{align*}
  0
  &=\int_\Omega \bigl(q'(u)h(u)g^2+q(u)h'(u)g^2
     +q(u)h(u)\nabla g\cdot\nabla u
     +q(u)h(u)g\Delta u\bigr)\,dx \\
  &=\int_\Omega \bigl(q'(u)h(u)g^2
     +q(u)\bigl(h'(u)g^2+a h(u)g-\lambda u h(u)g\bigr)\bigr)\,dx,
\end{align*}
where we used \(\Delta u=-\lambda u\) and \eqref{eq:a-identities}. Moving
the term containing \(q'h\) to the other side gives
\[
  \int_\Omega q'(u)h(u)g^2\,dx
  =\int_\Omega q(u)\bigl(\lambda u h(u)-g h'(u)-a h(u)\bigr)g\,dx,
\]
which is \eqref{eq:adjoint}.
\end{proof}

For the proof of the estimate for \(\mu_3\), we need the following weighted
consequences of the Bochner identity. Convexity enters through the sign of
the Neumann boundary term after integration by parts. We include the
derivation because the precise constants are used in the matrix argument
below.
\begin{lemma}\label{lem:weighted-bochner-mu3}
The following identities and inequalities hold:
\begin{equation}\label{eq:bochner-mu3}
  \frac12\Delta g=|M|^2-\lambda g,
\end{equation}
\begin{equation}\label{eq:boundary-g-mu3}
  \partial_\nu g\le0\quad\text{on }\partial\Omega,
\end{equation}
and consequently
\begin{equation}\label{eq:bochner-weighted}
  \int_\Omega g^2|M|^2\,dx
  +4\int_\Omega g|M\nabla u|^2\,dx
  \le \lambda\int_\Omega g^3\,dx.
\end{equation}
Moreover, if \(h\) is affine, then
\begin{equation}\label{eq:bochner-h}
  \int_\Omega h(u)^2g|M|^2\,dx
  +2\int_\Omega h(u)^2|M\nabla u|^2\,dx
  \le
  \lambda\int_\Omega h(u)^2g^2\,dx
  -h'\int_\Omega h(u)a g^2\,dx.
\end{equation}
\end{lemma}

\begin{proof}
The Bochner identity \eqref{eq:bochner-mu3} follows from
\[
  \frac12\Delta|\nabla u|^2=|D^2u|^2+\nabla u\cdot\nabla\Delta u
\]
and \(\Delta u=-\lambda u\).

On \(\partial\Omega\), the Neumann condition gives \(\nabla u=\nabla_Tu\).
If \(\tau\) is tangent to \(\partial\Omega\), differentiating
\(\nabla u\cdot\nu=0\) in the direction \(\tau\) gives
\[
  D^2u(\tau,\nu)+\II(\tau,\nabla_Tu)=0.
\]
Taking \(\tau=\nabla_Tu\), we get
\[
  \partial_\nu g=2D^2u(\nu,\nabla u)
  =-2\II(\nabla_Tu,\nabla_Tu)\le0,
\]
which proves \eqref{eq:boundary-g-mu3}.

Multiplying \eqref{eq:bochner-mu3} by \(g^2\), integrating by parts, and
using \eqref{eq:boundary-g-mu3}, we find
\begin{align*}
  \int_\Omega g^2|M|^2\,dx-
  \lambda\int_\Omega g^3\,dx
  &=\frac12\int_\Omega g^2\Delta g\,dx \\
  &\le -\int_\Omega g|\nabla g|^2\,dx
  =-4\int_\Omega g|M\nabla u|^2\,dx,
\end{align*}
which is \eqref{eq:bochner-weighted}.

Finally let \(h\) be affine. Multiplying \eqref{eq:bochner-mu3} by
\(h(u)^2g\), integrating by parts, and again using
\eqref{eq:boundary-g-mu3}, we obtain
\begin{align*}
  \int_\Omega h(u)^2g|M|^2\,dx
  -\lambda\int_\Omega h(u)^2g^2\,dx
  &=\frac12\int_\Omega h(u)^2g\Delta g\,dx \\
  &\le -\frac12\int_\Omega \nabla\bigl(h(u)^2g\bigr)\cdot\nabla g\,dx \\
  &=-h'\int_\Omega h(u)ag^2\,dx
    -2\int_\Omega h(u)^2|M\nabla u|^2\,dx,
\end{align*}
where we used \(\nabla g=2M\nabla u\) and
\(\nabla g\cdot\nabla u=ag\). Rearranging proves \eqref{eq:bochner-h}.
\end{proof}

The key estimate is the following weighted polynomial inequality.
\begin{proposition}\label{prop:weighted}
For every polynomial \(q\) of degree at most \(2\),
\begin{equation}\label{eq:weighted-quadratic}
  \int_\Omega q'(u)^2g^2\,dx
  \le
  8\lambda\int_\Omega q(u)^2g\,dx.
\end{equation}
\end{proposition}

We prove \cref{prop:weighted} in two steps. Since \(u\) is nonconstant,
\[
  S_2=\int_\Omega g^2\,dx>0.
\]
Set
\begin{equation}\label{eq:moment-notation-mu3}
  \alpha=\frac{\int_\Omega ug^2\,dx}{S_2},
  \qquad
  \eta=u-\alpha,
  \qquad
  Q=\int_\Omega \eta^2g^2\,dx,
  \qquad
  G=\int_\Omega g^3\,dx.
\end{equation}
Then \(\int_\Omega \eta g^2\,dx=0\). The first step in proving the
weighted polynomial estimate is to control the centered moments of the
measure \(g^2\,dx\).

\begin{lemma}\label{lem:moments-mu3}
With the notation \eqref{eq:moment-notation-mu3}, one has \(Q>0\),
\(G>0\),
\begin{equation}\label{eq:G-bound-mu3}
  G\le5\lambda Q,
\end{equation}
and
\begin{equation}\label{eq:alpha-bound-mu3}
  \lambda\alpha^2S_2
  \le
  \frac{24}{5}\bigl(5\lambda Q-G\bigr).
\end{equation}
\end{lemma}

\begin{proof}
If \(Q=0\), then the continuous nonnegative function \(\eta^2g^2\)
vanishes everywhere. The open set \(\{g>0\}\) is nonempty, since otherwise
\(u\) would be constant. On \(\{g>0\}\), the identity
\(\eta^2g^2=0\) would give \(u=\alpha\), hence \(\nabla u=0\), a
contradiction. Thus \(Q>0\). Also \(G>0\).

Because the normal component of \(g^2\nabla u\) vanishes on the boundary,
\[
  0=\int_\Omega\operatorname{div}(g^2\nabla u)\,dx.
\]
Using \(\nabla g\cdot\nabla u=2W\) and \(\Delta u=-\lambda u\), we get
\begin{equation}\label{eq:moment-id-0-mu3}
  \int_\Omega gW\,dx=\frac{\lambda\alpha S_2}{4}.
\end{equation}
Similarly,
\[
  0=\int_\Omega\operatorname{div}(\eta g^2\nabla u)\,dx.
\]
Since \(\nabla \eta=\nabla u\), \(\int_\Omega \eta g^2\,dx=0\), and
\(\int_\Omega \eta u g^2\,dx=Q\), this gives
\begin{equation}\label{eq:moment-id-1-mu3}
  \int_\Omega \eta gW\,dx=\frac{\lambda Q-G}{4}.
\end{equation}
The functions
\[
  \phi_0=\frac{g}{\sqrt{S_2}},
  \qquad
  \phi_1=\frac{\eta g}{\sqrt Q}
\]
are orthonormal in \(L^2(\Omega)\). Bessel's inequality, together with
\eqref{eq:moment-id-0-mu3} and \eqref{eq:moment-id-1-mu3}, yields
\begin{equation}\label{eq:bessel-mu3}
  \frac{\lambda^2\alpha^2S_2}{16}
  +\frac{(\lambda Q-G)^2}{16Q}
  \le
  \int_\Omega W^2\,dx.
\end{equation}
By \eqref{eq:bochner-weighted} and the pointwise inequality
\(g^2|M|^2\ge g|M\nabla u|^2\),
\begin{equation}\label{eq:five-bound-mu3}
  5\int_\Omega g|M\nabla u|^2\,dx\le\lambda G.
\end{equation}
Also \(W^2\le g|M\nabla u|^2\) pointwise. Combining this with
\eqref{eq:bessel-mu3} and \eqref{eq:five-bound-mu3} gives
\begin{equation}\label{eq:key-z-pre-mu3}
  \frac{\lambda^2\alpha^2S_2}{16}
  +\frac{(\lambda Q-G)^2}{16Q}
  \le
  \frac{\lambda G}{5}.
\end{equation}
Set \(z=G/(\lambda Q)>0\). After dividing
\eqref{eq:key-z-pre-mu3} by \(\lambda^2Q/16\), we obtain
\begin{equation}\label{eq:z-ineq-mu3}
  \frac{\alpha^2S_2}{Q}
  \le
  \frac{16}{5}z-(1-z)^2
  =(5-z)\left(z-\frac15\right).
\end{equation}
The right-hand side is nonnegative, so \(1/5\le z\le5\). The upper bound is
\eqref{eq:G-bound-mu3}. Since \(z\le5\),
\[
  z-\frac15\le\frac{24}{5}.
\]
Therefore \eqref{eq:z-ineq-mu3} gives
\[
  \alpha^2S_2\le\frac{24}{5}Q(5-z).
\]
Multiplication by \(\lambda\) yields \eqref{eq:alpha-bound-mu3}.
\end{proof}

The second step is an affine estimate for the first-order operator \(R\).
\begin{lemma}\label{lem:smu3}
For every affine function \(h\),
\begin{equation}\label{eq:affine-estimate-mu3}
  \|Rh\|_{L^2(g)}^2\le8\lambda\|h\|_{L^2(g^2)}^2.
\end{equation}
\end{lemma}

\begin{proof}
Every affine function can be written uniquely as \(h=h_v\), where, for
\(v=(\beta,\gamma)^{\top}\),
\[
  h_v(t)=\beta(t-\alpha)+\gamma,
  \qquad
  h_v(u)=\beta \eta+\gamma,
  \qquad
  \Phi=\binom{\eta}{1}.
\]
Thus \(h_v(u)=v^{\top}\Phi\) and \(h_v'=\beta\). All matrix
inequalities below are inequalities of symmetric quadratic forms on
\(\R^2\). Define
\begin{align*}
  \mathbf U&=\int_\Omega \lambda^2u^2\Phi\Phi^{\top}g\,dx,\\
  \mathbf K&=\int_\Omega \lambda ua\,\Phi\Phi^{\top}g\,dx,\\
  \mathbf V&=\int_\Omega a^2\Phi\Phi^{\top}g\,dx.
\end{align*}
For arbitrary \(p,r\in\R^2\),
\[
  \binom{p}{r}^{\top}
  \begin{pmatrix}
    \mathbf U&\mathbf K\\
    \mathbf K&\mathbf V
  \end{pmatrix}
  \binom{p}{r}
  =\int_\Omega (\lambda u h_p+a h_r)^2g\,dx\ge0.
\]
Hence
\begin{equation}\label{eq:gram-block-mu3}
  \begin{pmatrix}
    \mathbf U&\mathbf K\\
    \mathbf K&\mathbf V
  \end{pmatrix}\succeq0.
\end{equation}

We first express \(\mathbf K\) in terms of \(\mathbf U\). Expanding
\[
  0=\int_\Omega\operatorname{div}(u h_v^2 g\nabla u)\,dx
\]
gives
\begin{equation}\label{eq:K-identity-scalar-mu3}
  \int_\Omega uah_v^2g\,dx
  =
  \lambda\int_\Omega u^2h_v^2g\,dx
  -\int_\Omega h_v^2g^2\,dx
  -2\beta\int_\Omega u h_v g^2\,dx.
\end{equation}
The centering condition \(\int_\Omega \eta g^2\,dx=0\) implies
\begin{align}
  \int_\Omega h_v^2g^2\,dx&=\beta^2Q+\gamma^2S_2,
  \label{eq:A-moment-mu3}\\
  \int_\Omega u h_v g^2\,dx&=\beta Q+\alpha\gamma S_2.
  \label{eq:T-moment-mu3}
\end{align}
Multiplying \eqref{eq:K-identity-scalar-mu3} by \(\lambda\) and using the
definitions of \(\mathbf U\) and \(\mathbf K\), we get, for every \(v\),
\[
  v^{\top}\mathbf K v
  =v^{\top}\mathbf U v
  -\lambda\left(
    \int_\Omega h_v^2g^2\,dx
    +2\beta\int_\Omega u h_v g^2\,dx
  \right).
\]
Therefore \eqref{eq:A-moment-mu3} and \eqref{eq:T-moment-mu3} imply the
quadratic-form identity
\begin{equation}\label{eq:KUD-mu3}
  \mathbf K=\mathbf U-\mathbf D,
\end{equation}
where
\begin{equation}\label{eq:D-matrix-mu3}
  \mathbf D
  =\lambda
  \begin{pmatrix}
    3Q&\alpha S_2\\
    \alpha S_2&S_2
  \end{pmatrix}.
\end{equation}

We next obtain an upper bound for \(\mathbf V\). By \eqref{eq:bochner-h}
with \(h=h_v\), and since \(g|M|^2\ge |M\nabla u|^2\),
\[
  3\int_\Omega h_v^2|M\nabla u|^2\,dx
  \le
  \lambda\int_\Omega h_v^2g^2\,dx
  -\beta\int_\Omega h_vag^2\,dx.
\]
Using \(a^2g\le4|M\nabla u|^2\), we obtain
\begin{equation}\label{eq:V-first-mu3}
  v^{\top}\mathbf Vv
  \le
  \frac43\left(
    \lambda\int_\Omega h_v^2g^2\,dx
    -\beta\int_\Omega h_vag^2\,dx
  \right).
\end{equation}
On the other hand,
\[
  0=\int_\Omega\operatorname{div}(h_vg^2\nabla u)\,dx
\]
gives
\begin{equation}\label{eq:ha-identity-mu3}
  2\int_\Omega h_vag^2\,dx
  =\lambda\int_\Omega u h_v g^2\,dx-\beta G.
\end{equation}
Substituting \eqref{eq:A-moment-mu3}, \eqref{eq:T-moment-mu3}, and
\eqref{eq:ha-identity-mu3} into \eqref{eq:V-first-mu3}, we obtain
\begin{equation}\label{eq:V-B-mu3}
  \mathbf V\preceq\mathbf B,
\end{equation}
where
\begin{equation}\label{eq:B-matrix-mu3}
  \mathbf B=
  \begin{pmatrix}
    \dfrac23(\lambda Q+G)&-\dfrac13\lambda\alpha S_2\\[1ex]
    -\dfrac13\lambda\alpha S_2&\dfrac43\lambda S_2
  \end{pmatrix}.
\end{equation}

By \eqref{eq:gram-block-mu3}, \eqref{eq:KUD-mu3}, and
\eqref{eq:V-B-mu3},
\begin{equation}\label{eq:enlarged-block-mu3}
  \begin{pmatrix}
    \mathbf U&\mathbf U-\mathbf D\\
    \mathbf U-\mathbf D&\mathbf B
  \end{pmatrix}\succeq0.
\end{equation}
Indeed, this matrix is obtained from the positive semidefinite block matrix
in \eqref{eq:gram-block-mu3} by using \(\mathbf K=\mathbf U-\mathbf D\) and
then adding the positive semidefinite block matrix
\[
  \begin{pmatrix}
    0&0\\
    0&\mathbf B-\mathbf V
  \end{pmatrix}.
\]
Testing \eqref{eq:enlarged-block-mu3} on \((v,v/2)\in\R^2\times\R^2\)
gives
\[
  0\le 2v^{\top}\mathbf Uv-v^{\top}\mathbf Dv
  +\frac14v^{\top}\mathbf Bv,
\]
which is equivalent to
\begin{equation}\label{eq:U-lower-mu3}
  \mathbf U\succeq\frac12\mathbf D-\frac18\mathbf B.
\end{equation}

Define
\[
  \mathcal E[h_v]
  =8\lambda\int_\Omega h_v^2g^2\,dx
  -\int_\Omega (Rh_v)^2g\,dx.
\]
Let
\[
  A_v=\int_\Omega h_v^2g^2\,dx,
  \qquad
  T_v=\int_\Omega u h_vg^2\,dx,
  \qquad
  J_v=\int_\Omega h_vag^2\,dx.
\]
From \eqref{eq:K-identity-scalar-mu3},
\[
  v^{\top}\mathbf Dv=\lambda(A_v+2\beta T_v),
\]
and \eqref{eq:ha-identity-mu3} says
\(2J_v=\lambda T_v-\beta G\). Expanding
\(Rh_v=\lambda uh_v-\beta g-ah_v\), we get
\begin{align*}
  \int_\Omega(Rh_v)^2g\,dx
  &=v^{\top}\mathbf Uv+
    \beta^2G+v^{\top}\mathbf Vv
    -2\lambda\beta T_v-2v^{\top}\mathbf Kv+2\beta J_v\\
  &=-v^{\top}\mathbf Uv+v^{\top}\mathbf Vv
    -\lambda\beta T_v+2v^{\top}\mathbf Dv.
\end{align*}
Consequently,
\begin{equation}\label{eq:defect-matrix-mu3}
  \mathcal E[h_v]
  =v^{\top}(\mathbf U-\mathbf V+\mathbf E_0)v,
\end{equation}
where
\begin{equation}\label{eq:E0-matrix-mu3}
  \mathbf E_0=
  \begin{pmatrix}
    3\lambda Q&-\dfrac32\lambda\alpha S_2\\[1ex]
    -\dfrac32\lambda\alpha S_2&6\lambda S_2
  \end{pmatrix}.
\end{equation}
Using \eqref{eq:defect-matrix-mu3}, \eqref{eq:U-lower-mu3}, and \eqref{eq:V-B-mu3}, we get
\begin{equation}\label{eq:defect-lower-mu3}
  \mathcal E[h_v]
  \ge
  v^{\top}\left(\frac12\mathbf D-\frac98\mathbf B+\mathbf E_0\right)v.
\end{equation}
Using \eqref{eq:D-matrix-mu3}, \eqref{eq:B-matrix-mu3}, and
\eqref{eq:E0-matrix-mu3}, the entries of
\(\frac12\mathbf D-\frac98\mathbf B+\mathbf E_0\) are
\begin{align*}
  (1,1):\quad&
  \frac32\lambda Q-\frac34(\lambda Q+G)+3\lambda Q
  =\frac34(5\lambda Q-G),\\
  (1,2):\quad&
  \frac12\lambda\alpha S_2+\frac38\lambda\alpha S_2
  -\frac32\lambda\alpha S_2
  =-\frac58\lambda\alpha S_2,\\
  (2,2):\quad&
  \frac12\lambda S_2-\frac32\lambda S_2+6\lambda S_2
  =5\lambda S_2.
\end{align*}
Thus
\begin{equation}\label{eq:L-matrix-mu3}
  \frac12\mathbf D-\frac98\mathbf B+\mathbf E_0
  =\mathbf L
  :=
  \begin{pmatrix}
    \dfrac34D_0&-\dfrac58\lambda\alpha S_2\\[1ex]
    -\dfrac58\lambda\alpha S_2&5\lambda S_2
  \end{pmatrix},
  \qquad
  D_0=5\lambda Q-G.
\end{equation}
By \eqref{eq:L-matrix-mu3} and \cref{lem:moments-mu3}, \(D_0\ge0\) and therefore
\[
  \det\mathbf L
  =\frac{15}{4}\lambda S_2D_0
    -\frac{25}{64}\lambda^2\alpha^2S_2^2
  \ge
    \left(\frac{15}{4}-\frac{25}{64}\cdot\frac{24}{5}\right)
    \lambda S_2D_0
  =\frac{15}{8}\lambda S_2D_0\ge0.
\]
Both diagonal entries of \(\mathbf L\) are nonnegative, so
\(\mathbf L\succeq0\). It follows from \eqref{eq:defect-lower-mu3} that
\(\mathcal E[h_v]\ge0\), which is \eqref{eq:affine-estimate-mu3}.
\end{proof}

\begin{proof}[Proof of \cref{prop:weighted}]
Let \(q\) have degree at most \(2\), and set \(h=q'\). Then \(h\) is
affine. \Cref{lem:adjoint} gives
\[
  \|h\|_{L^2(g^2)}^2
  =\int_\Omega q'(u)h(u)g^2\,dx
  =\int_\Omega q(u)Rh\,g\,dx.
\]
By the Cauchy--Schwarz inequality and \cref{lem:smu3},
\[
  \|h\|_{L^2(g^2)}^4
  \le \|q\|_{L^2(g)}^2\|Rh\|_{L^2(g)}^2
  \le 8\lambda\|q\|_{L^2(g)}^2\|h\|_{L^2(g^2)}^2.
\]
If \(\|h\|_{L^2(g^2)}=0\), then \eqref{eq:weighted-quadratic} is immediate.
Otherwise, division by \(\|h\|_{L^2(g^2)}^2\) proves \eqref{eq:weighted-quadratic}.
\end{proof}

We now apply the weighted polynomial inequality to the cubic test space
generated by a first eigenfunction.
\begin{proposition}\label{prop:mu3-smooth}
Let \(N\ge2\), and let \(\Omega\subset\R^N\) be a smooth bounded convex domain.
Then $\mu_3(\Omega)\le9\mu_1(\Omega)$.
\end{proposition}

\begin{proof}
Let \(P\) be a polynomial of degree at most \(3\), and set \(f=P(u)\).
Then
\begin{equation}\label{eq:laplacian-P}
  -\Delta f=\lambda uP'(u)-gP''(u),
  \qquad
  \partial_\nu f=0.
\end{equation}
The following identity is exact:
\begin{equation}\label{eq:laplace-identity}
  \|{-\Delta P(u)}\|_{L^2(\Omega)}^2
  =
  \lambda\int_\Omega P'(u)^2g\,dx
  +\int_\Omega P''(u)^2g^2\,dx.
\end{equation}
Indeed, the boundary flux of \(uP'(u)^2\nabla u\) vanishes, and hence
\[
  0=\int_\Omega\operatorname{div}\bigl(uP'(u)^2\nabla u\bigr)\,dx.
\]
After expansion,
\begin{equation}\label{eq:cancellation}
  \lambda\int_\Omega u^2P'(u)^2\,dx
  =
  \int_\Omega P'(u)^2g\,dx
  +2\int_\Omega uP'(u)P''(u)g\,dx.
\end{equation}
Expanding the square in \eqref{eq:laplacian-P} and using
\eqref{eq:cancellation} proves \eqref{eq:laplace-identity}.

Apply \cref{prop:weighted} to \(q=P'\). Since \(P'\) has degree at most
\(2\),
\[
  \int_\Omega P''(u)^2g^2\,dx
  \le
  8\lambda\int_\Omega P'(u)^2g\,dx.
\]
Thus \eqref{eq:laplace-identity} yields
\begin{equation}\label{eq:Delta-bound}
  \|{-\Delta f}\|_{L^2(\Omega)}^2
  \le
  9\lambda\int_\Omega |\nabla f|^2\,dx.
\end{equation}
Let
\[
  \bar f=\frac1{|\Omega|}\int_\Omega f\,dx.
\]
Since \(\partial_\nu f=0\), the function \(-\Delta f\) is orthogonal to
constants, and integration by parts gives
\[
  \int_\Omega |\nabla f|^2\,dx
  =\langle -\Delta f,f-\bar f\rangle_{L^2(\Omega)}.
\]
The Cauchy--Schwarz inequality and \eqref{eq:Delta-bound} imply
\[
  \left(\int_\Omega |\nabla f|^2\,dx\right)^2
  \le
  9\lambda
  \left(\int_\Omega |\nabla f|^2\,dx\right)
  \int_\Omega |f-\bar f|^2\,dx.
\]
If the Dirichlet energy vanishes, the desired estimate is trivial.
Otherwise, division by that energy gives
\begin{equation}\label{eq:rayleigh-cubic}
  \int_\Omega |\nabla f|^2\,dx
  \le
  9\lambda\int_\Omega |f-\bar f|^2\,dx
  \le
  9\lambda\int_\Omega f^2\,dx.
\end{equation}

The functions \(1,u,u^2,u^3\) are linearly independent. Indeed, if a
polynomial \(P\) of degree at most \(3\) satisfies \(P(u)=0\) almost
everywhere, continuity gives \(P(u)=0\) everywhere. Since \(\Omega\) is
connected and \(u\) is continuous and nonconstant, \(u(\Omega)\) is a
nondegenerate interval. Hence \(P\) has infinitely many zeros and is the
zero polynomial. Therefore $\mathcal V=\Span\{1,u,u^2,u^3\}$
has dimension \(4\). Estimate \eqref{eq:rayleigh-cubic} holds for every
\(f\in\mathcal V\). By the min--max principle,
\[
  \mu_3(\Omega)
  \le
  \sup_{0\ne f\in\mathcal V}
  \frac{\int_\Omega |\nabla f|^2\,dx}{\int_\Omega f^2\,dx}
  \le9\lambda=9\mu_1(\Omega).
\qedhere\]
\end{proof}

\section{Extension to nonsmooth convex domains and sharpness}\label{sec:approx-sharp}
The preceding two sections prove the desired estimates under the additional
assumption that the boundary is smooth. We now remove this assumption by
combining smooth approximation of convex bodies with spectral continuity for
Neumann eigenvalues on convex domains.

\begin{lemma}\label{lem:approximation}
Let \(N\ge2\), \(k\ge1\), and \(C>0\). Suppose that
\[
  \mu_k(D)\le C\mu_1(D)
\]
holds for every smooth bounded convex domain \(D\subset\R^N\). Then it
holds for every nonempty bounded open convex set \(\Omega\subset\R^N\).
\end{lemma}

\begin{proof}

Let \(K=\overline\Omega\). Since \(\Omega\) is open and convex,
\(\Omega=\operatorname{int}K\), and \(K\) is a convex body with nonempty
interior. By smooth approximation of convex bodies, obtained for instance
by regularizing support functions and then adding a vanishing Euclidean ball
if necessary, one can choose convex bodies \(K_j\) with \(C^\infty\) strictly
convex boundaries such that
\[
  K\subset K_j,
  \qquad
  K_j\to K
  \quad\text{in the Hausdorff metric};
\]
see, for instance, \cite[Section~3.4]{Schneider2014}. Set
\(\Omega_j=\operatorname{int}K_j\).

Choose an interior point of \(\Omega\) as the origin. Then there exists
\(r>0\) such that \(\overline B_r(0)\subset K\subset K_j\) for all \(j\). For all sufficiently large \(j\), there is \(R_1<\infty\) such that
\(K\cup K_j\subset\overline B_{R_1}(0)\). Put
\(\delta_j=d_H(K_j,K)\), where \(d_H\) denotes the Hausdorff distance between compact subsets of
\(\mathbb R^N\). Since
\[
  K_j\subset K+\delta_j\overline B_1(0)
  \quad\text{and}\quad
  \delta_j\overline B_1(0)\subset \frac{\delta_j}{r}K,
\]
and since \(0\in K\) and \(K\) is convex, so that
\(K+tK=(1+t)K\) for every \(t\ge0\), we have
\[
  K_j\subset K+\frac{\delta_j}{r}K
  =\left(1+\frac{\delta_j}{r}\right)K.
\]
Also \(K\subset K_j\), so if \(\rho_j,\rho:\mathbb S^{N-1}\to(0,\infty)\)
are the radial functions of \(K_j\) and \(K\), then
\[
  \rho(\theta)\le\rho_j(\theta)
  \le\left(1+\frac{\delta_j}{r}\right)\rho(\theta)
  \quad\text{for every }\theta\in\mathbb S^{N-1}.
\]
Using the common outer bound \(R_1\), we obtain $\|\rho_j-\rho\|_{L^\infty(\mathbb S^{N-1})}\le R_1\delta_j/r\longrightarrow0$.

The radial functions are uniformly bounded below by \(r\) and, for all
large \(j\), uniformly bounded above by \(R_1\). Hence Ross's Lipschitz
continuity theorem for Neumann eigenvalues on convex domains in the
radial-function parametrization applies \cite[Theorem~4.2]{Ross2004}.
Taking into account that Ross indexes the Neumann spectrum starting with
\(\mu_1=0\), while our convention is \(\mu_0=0\), we obtain, for each fixed
\(m\ge0\),
\[
  |\mu_m(\Omega_j)-\mu_m(\Omega)|
  \le C_m\|\rho_j-\rho\|_{L^\infty(\mathbb S^{N-1})}
  \longrightarrow 0.
\]
Passing to the limit in
\(\mu_k(\Omega_j)\le C\mu_1(\Omega_j)\) yields the desired estimate.
\end{proof}

We now combine the smooth estimates with the approximation lemma and then
check sharpness.
\begin{proof}[Proof of \cref{thm:mu2} and \cref{thm:mu3}]
If \(N=1\), then \(\Omega\) is an interval \((0,L)\), and $\mu_k((0,L))=\frac{k^2\pi^2}{L^2}$, $k=0,1,2,\ldots$.
Thus both inequalities are equalities.

Assume \(N\ge2\). For smooth bounded convex domains,
\eqref{eq:main-mu2} follows from \cref{prop:mu2-smooth}, and
\eqref{eq:main-mu3} follows from \cref{prop:mu3-smooth}. The general
bounded convex case follows by \cref{lem:approximation}, first with
\((k,C)=(2,4)\) and then with \((k,C)=(3,9)\).

It remains to prove sharpness. For \(N\ge2\), consider the rectangular boxes $\Omega_{L,\varepsilon}=(0,L)\times(0,\varepsilon)^{N-1}$, where $0<\varepsilon<L/3$.
Separation of variables gives the Neumann spectrum
\begin{equation}\label{eq:box-spectrum}
  \pi^2\left(
    \frac{n_1^2}{L^2}+
    \frac{n_2^2+\cdots+n_N^2}{\varepsilon^2}
  \right),
  \qquad (n_1,\ldots,n_N)\in\mathbb N_0^N.
\end{equation}
The assumption \(\varepsilon<L/3\) implies
\(\pi^2/\varepsilon^2>9\pi^2/L^2\). Thus every mode with at least one
transverse index \(n_i\ne0\), \(i\ge2\), lies above
\(9\pi^2/L^2\), while the longitudinal modes
\((n_1,n_2,\ldots,n_N)=(0,\ldots,0),(1,0,\ldots,0),(2,0,\ldots,0),
(3,0,\ldots,0)\) give the first four values. Therefore the first four values in \eqref{eq:box-spectrum}, counted with
multiplicity, are
\[
  0,\quad \frac{\pi^2}{L^2},\quad \frac{4\pi^2}{L^2},\quad
  \frac{9\pi^2}{L^2}.
\]
It follows that
\[
  \frac{\mu_2(\Omega_{L,\varepsilon})}{\mu_1(\Omega_{L,\varepsilon})}=4,
  \qquad
  \frac{\mu_3(\Omega_{L,\varepsilon})}{\mu_1(\Omega_{L,\varepsilon})}=9.
\]
This proves optimality in every dimension. If one restricts the sharpness
statement to smooth strictly convex domains, the same values are approached
by smooth strictly convex approximations and the spectral convergence argument used in \Cref{lem:approximation}.
\end{proof}

\section*{Acknowledgments and AI disclosure}

During the development and preparation of this work, the authors used a
generative AI tool for preliminary, non-authoritative exploratory assistance,
including organizational discussion and the consideration of possible
approaches.  The AI-generated outputs were not treated as mathematical
sources, and no argument was included without independent verification and
substantial revision by the authors.  All theorem statements, proofs,
computations, references, and final arguments were independently checked,
revised, and finalized by the authors, who take full responsibility for the
correctness, originality, and integrity of the paper.

\end{document}